\documentclass{amsart}
\usepackage{mathscinet}

\usepackage[all]{xy}
\usepackage{graphics,bm}
\usepackage{color}
\usepackage[T1]{fontenc}
\usepackage{parskip}
\usepackage[pagebackref, colorlinks]{hyperref}
\usepackage{enumitem}
\usepackage{comment}
\usepackage{mathtools}

\usepackage{tikz,tikz-cd}
\usetikzlibrary{decorations.markings}

\usepackage[colorinlistoftodos,prependcaption]{todonotes}

\newcommand{\B}[1]{{\mathbb #1}}

\newcommand{\MCG}{\mathcal{M}}
\newcommand{\Diff}{\operatorname{Diff}}
\newcommand{\Homeo}{\operatorname{Homeo}}

\newtheorem{theorem}{Theorem}[section]
\newtheorem{theorem*}{Theorem}
\newtheorem{lemma}[theorem]{Lemma}

\newtheorem{corollary}[theorem]{Corollary}

\theoremstyle{definition}
\newtheorem{example}[theorem]{Example}

\newtheorem{remark}[theorem]{Remark}
\newtheorem*{remark*}{Remark}
\newtheorem*{remarks*}{Remarks}
\newtheorem*{corollary*}{Corollary}

\numberwithin{figure}{section}
\numberwithin{table}{section}
\numberwithin{equation}{section}

\def\B{\mathbf}

\newcommand{\la}{\langle}
\newcommand{\ra}{\rangle}

\newcommand{\OP}{\operatorname}
\newcommand{\MP}{{\scriptscriptstyle \mathcal{M}_+}}
\newcommand{\M}{{\scriptscriptstyle \mathcal{M}}}

\newcommand{\wt}{\widetilde}

\newcommand{\n}{|\hspace{-1px}|}

\newcommand{\Gb}{\Gamma_{\hspace{-2px}b}}

\begin{document}

\title[Volume and Euler classes in transformation groups]{Volume and Euler classes in bounded cohomology of transformation groups}
\author{Michael Brandenbursky}\thanks{M.B. was partially supported by the Israel Science Foundation grant 823/23.}
\address{Department of Mathematics, Ben Gurion University, Israel}
\email{brandens@bgu.ac.il} 
\author{Micha\l\ Marcinkowski}\thanks{M.M. was supported by Opus 2017/27/B/ST1/01467 funded by Narodowe Centrum Nauki.}
\address{Institute of Mathematics, Wroc\l aw University, Poland}
\email{marcinkow@math.uni.wroc.pl}

\begin{abstract} 
Let $M$ be an oriented
smooth manifold, and $\Homeo(M,\omega)$ the group of measure preserving homeomorphisms of $M$, where $\omega$ is a finite measure induced by a volume form. In this paper, we define volume and Euler classes in bounded cohomology of an infinite dimensional transformation group $\Homeo_0(M,\omega)$ and $\Homeo_+(M,\omega)$ respectively, and in several cases prove their non-triviality. More precisely, we define:
\begin{itemize}
\item Volume classes in $\OP{H}_b^n(\Homeo_0(M,\omega))$ where $M$ is a hyperbolic manifold of dimension $n$. 
\item Euler classes in $\OP{H}_b^2(\Homeo_+(S,\omega))$ where $S$ is an oriented closed hyperbolic surface.
\end{itemize}
We show that Euler classes have positive norms for any closed hyperbolic surface and volume classes have positive norms for all hyperbolic surfaces and certain hyperbolic $3$-manifolds; hence, they are non-trivial.
\end{abstract}

\maketitle
 
\section{Introduction}

Let $M$ be an oriented connected smooth manifold. Suppose $\omega$ is a finite measure induced by a volume form on~$M$.
In \cite{BM} we defined a homomorphism
$$
\Gb \colon \OP{H}_b^\bullet(\pi_1(M)) \to \OP{H}_b^\bullet(\OP{Homeo}_0(M,\omega)),
$$

where $\OP{H}_b^\bullet(\OP{Homeo}_0(M,\omega))$ is the bounded cohomology of a discrete group. The map $\Gamma_b$ is a generalization of a map defined by Gambaudo and Ghys in the quasimorphism setting \cite[Section 5]{MR2104597}. 
 
$\Gb$ was used in \cite{BM} to show that the $3^{rd}$ bounded cohomology of $\OP{Homeo}_0(M,\omega)$ is infinite dimensional for many manifolds $M$. In \cite{nitsche} more results concerning bounded cohomology, as well as standard cohomology, of $\OP{Homeo}_0(M,\omega)$ were obtained. Variations of $\Gamma_b$ were used in \cite{kimura} to prove similar results concerning $\OP{Diff}_0(S,area)$ where $S$ is a disc, sphere or torus.  

In this paper, we continue this line of research and focus on two important families of cohomology classes described below. Moreover, we construct an extension of $\Gamma_b$ to a map
$$
 \Gamma_b^\M \colon \OP{H}_b^\bullet(\MCG(M,*)) \to \OP{H}_b^\bullet(\Homeo(M,\omega)),
$$
where $\MCG(M,*)$ is the mapping class group of once punctured $M$, see Section~\ref{ss:gamma^m}. This extension is used to define the Euler class on $\Homeo_+(S,\omega)$ for an oriented closed hyperbolic surface $S$ and we hope that it might be useful to study the cohomology of $\Homeo(M,\omega)$ for other manifolds $M$.

\textbf{The bounded volume class.} Let $M$ be an oriented hyperbolic manifold and let $\omega_h$ be the volume form induced by the hyperbolic metric. 
In this setting $\omega_h$ defines a class $Vol_M \in \OP{H}_b^n(\pi_1(M)) \simeq \OP{H}_b^n(M)$, see Section \ref{ss:vol}. If $M$ is closed, $Vol_M$ is a natural bounded representative of $[\omega_h] \in \OP{H}_{dR}^n(M)$.

Our main result positively answers the question in Section 5 of \cite{BM} for degree 2 and partially for degree 3. Moreover, it may be seen as a successful attempt to define a volume class in the bounded cohomology of an infinite dimensional transformation group.

\begin{theorem}\label{t:intro}
Let $M$ be an oriented manifold of dimension $n$ such that it is either:
\begin{itemize}
    \item A hyperbolic surface with a non-abelian fundamental group or
    \item A complete $3$-dimensional hyperbolic manifold that fibers over the circle with a non-compact fiber.
\end{itemize}
Suppose a measure $\omega$ is induced by a volume form on $M$ and $\omega$ is finite. Then the class $\Gb(Vol_M) \in \OP{H}_b^n(\OP{Homeo}_0(M,\omega))$ has positive norm and hence is non-trivial.
\end{theorem}

The class $Vol_M \in \OP{H}_b^n(M)$ was considered by Gromov and Thurston in the proof of the proportionality principle \cite{gromov82, thurston}. Moreover, it serves as a rich source for classes in $3^{rd}$ bounded cohomology of free and surface groups \cite{soma}.

It is worth mentioning that it is not known if the bounded cohomology in degree $n$ of a non-abelian free group for $n>3$ is non-trivial, hence our proof works in dimension~$2$ and sometimes in dimension $3$ so far. 
The problem of non-triviality of $\Gamma_b(Vol_M)$ for higher dimensional hyperbolic manifolds or closed hyperbolic $3$-manifolds is still open.

\textbf{The bounded Euler class.}
Let $S$ be a closed oriented surface and $\omega$ a measure defined by an area form on $S$. Denote by $\Homeo_+(S,\omega)$ the subgroup of $\Homeo(S,\omega)$ of orientation preserving homeomorphisms. 

Let $* \in S$ be a point in $S$ and $\MCG_+(S,*)$ be the orientation preserving mapping class group of $S~\backslash~\{*\}$. Recall that $\pi_1(S,*) < 
\MCG_+(S,*)$ due to the Birman exact sequence. 
As we will see in Section \ref{ss:gamma^m}, the map $\Gamma_b$ can be extended to a map $\Gamma^{\MP}_b$ such that the following diagram commutes:

$$
\begin{tikzcd}
\OP{H}_b^\bullet(\MCG_+(S,*)) \arrow{d}  \arrow{r}{\Gamma_b^{\MP}}& \OP{H}_b^\bullet(\Homeo_+(S,\omega)) \arrow{d}\\
\OP{H}_b^\bullet(\pi_1(S,*)) \arrow{r}{\Gamma_b}& \OP{H}_b^\bullet(\Homeo_0(S,\omega)) 
\end{tikzcd}
$$

The vertical arrows are induced by inclusions. 
Thus at the cost of passing to a bigger group $\MCG_+(S,*)$, we can generate classes in the group of $\omega$-preserving homeomorphisms of $S$ not necessarily isotopic to the identity. 

Let $\Homeo_+(S^1)$ be the group of orientation preserving homeomorphisms of the circle and let $e_b \in \OP{H}_b^2(\Homeo_+(S^1))$ be the bounded Euler class. 
There is a natural map $\alpha \colon \MCG_+(S,*) \to \Homeo_+(S^1)$ defined by the action of $\MCG_+(S,*) \simeq \OP{Aut}_+(\pi_1(S))$ on the Gromov boundary of $\pi_1(S)$.
Let $e_b^{\MP}$ and $e_b^S$ be the pull-backs of $e_b$ to $\MCG_+(S,*)$ and $\pi_1(S)$.

\begin{theorem}\label{t:intro2}
Let $S$ be a closed oriented surface of genus $\geq 2$, and $\omega$ a measure induced by an area form on~$S$.
Then the classes $\Gamma_b(e_b^S) \in \OP{H}_b^2(\Homeo_0(S,\omega))$
and $\Gamma_b^{\MP}(e_b^{\MP}) \in \OP{H}_b^2(\Homeo_+(S,\omega))$ have positive norms and hence are non-trivial.
\end{theorem}

We emphasise that in Theorem \ref{t:intro} and in Theorem \ref{t:intro2}
one can take instead of $\OP{Homeo}_0(M,\omega)$ smaller groups like $\OP{Diff}_0(M,\omega)$, or $\OP{Symp}_0(S,\omega)$ or $\OP{Ham}(S)$ whenever $S$ is a hyperbolic surface, and the same results hold. 

\textbf{On the proof and organization of the paper}. The method of the proof is a refinement (and at the same time a simplification) of the one from \cite{BM}. In the case of bounded volume classes,  
it is based on mapping a non-abelian free group $F$ to $\pi_1(M)$ and to $\OP{Homeo}_0(M,\omega)$, restricting $Vol_M$ and $\Gamma_b(Vol_M)$
to $F$ and then comparing these two classes.
An identical method is used for Euler classes. 
This technique is quite special, since we do not know many subgroups of $\OP{Homeo}_0(M,\omega)$. See \cite[Chapter 4]{mann.survey} for a survey on realizations of groups by diffeomorphisms and homeomorphisms.  

The outline of the paper is as follows. 
In Section \ref{ss:bounded.cohom} we give basic definitions of bounded cohomology. In Section \ref{ss:vol} we define two versions of the volume class: the topological $Vol_M \in \OP{H}_b^n(M)$ and its group version $Vol_M^{gp} \in \OP{H}_b^n(\pi_1(M))$. The topological version is well known and we use it in the proofs. We need the group version because $\Gamma_b$ takes classes from the group cohomology. As expected, $Vol_M$ and $Vol^{gp}_M$ define the same class under the canonical identification $\OP{H}_b^n(M) \simeq \OP{H}_b^n(\pi_1(M))$. We could not find a proof of this fact in the literature, thus for completeness we provide a detailed argument in Lemma \ref{l:group=top}.  In Section \ref{ss:euler.class} we define the Euler class in the mapping class group of a punctured surface. In Section \ref{s:definitions} we define the maps $\Gamma_b$ and $\Gamma_b^\M$. In Section \ref{s:restricting} we show that for hyperbolic surfaces and certain $3$-dimensional hyperbolic manifolds, the class $Vol_M$ restricts non-trivially to a free subgroup of $\pi_1(M)$.
The same result holds for Euler classes and closed surfaces of genus~$\geq 2$.
In Section \ref{s:proof} we give a simpler and at the same time more general version of a technical Lemma 3.1 from \cite{BM} and prove the main theorems. 
In Section \ref{s:dirac} we discuss the case when $\omega$ is the Dirac measure. 

\textbf{Acknowledgements.}
MB acknowledges the support of the Israeli Science Foundation grant 823/23. MB was partially supported by a Humboldt research fellowship. MM was supported by grant Opus 2017/27/B/ST1/01467 funded by the Narodowe Centrum Nauki. We thank the Center for Advanced Studies in Mathematics at Ben Gurion University for supporting the visit of the second author at BGU. The authors would like to thank Steve Farre for helpful discussions. 

\section{Preliminaries}\label{s:prelim}

\subsection{Bounded cohomology.}\label{ss:bounded.cohom}

Let us give the definitions of bounded cohomology of a group and a space.

Let $G$ be a~group. 
The space of bounded $n$-cochains is defined by
$$\OP{C}^n_b(G) = \{c \colon G^{n+1} \to \mathbb{R}~|~c~\text{is bounded}\}.$$
Let $d_n$ be the ordinary coboundary operator $d_n \colon  \OP{C}^n_b(G) \to \OP{C}^{n+1}_b(G)$. 
The group $G$ acts on 
$\OP{C}^n_b(G)$ by $$h(c)(g_0,\ldots,g_n) = c(h^{-1}g_0,\ldots,h^{-1}g_n) \hspace{0.4cm} \forall c \in \OP{C}^n_b(G), \hspace{0.4cm} \forall h, g_0,\ldots,g_n \in G.$$
Let $\OP{C}^n_b(G)^G$ be the space of $G$-invariant cochains.
The \textbf{bounded cohomology} of~$G$, denoted by $\OP{H}^\bullet_b(G)$, is the homology of the cochain complex $\{\OP{C}^n_b(G)^G,d_n\}$.
Note that $\OP{C}_b^n(G)^G$ is a~subcomplex of the space of all $G$-invariant cochains, hence we have a 
map $\OP{H}_b^n(G) \to \OP{H}^n(G,\mathbb{R})$ called the \textbf{comparison map}. 

On $\OP{C}^n_b(G)$ we have the supremum norm denoted by $\n\hspace{-3px}\cdot\hspace{-3px}\n$.
This norm induces a~semi-norm on $\OP{H}^n_b(G)$, i.e.,
if $C \in \OP{H}_b^n(G)$, then 
$$\n C \n = \OP{min}\{\n c \n~|~[c] = C\}.$$

Let $M$ be a topological space. 
A singular simplex is a continuous map from the standard simplex to $M$.
By $\OP{C}_n(M)$ we denote the space of singular chains and by $\OP{S}_n(M) \subset \OP{C}_n(M)$ the set of all singular simplices in $M$. 
Let $\OP{C}^n_b(M)$ be the set of linear functions from $\OP{C}_n(M)$ to the reals that are bounded on $\OP{S}_n(M)$.
Since $\OP{S}_n(M)$ generates $\OP{C}_n(M)$ 
one can think of an element in $\OP{C}_n^b(M)$ as a bounded 
function $c \colon \OP{S}_n(M) \to \mathbb{R}$. 

The \textbf{bounded cohomology} of $M$, denoted by $\OP{H}^\bullet_b(M)$, 
is the homology of the cochain complex $\{\OP{C}^n_b(M),d_n\}$, 
where $d_n \colon \OP{C}^n_b(M) \to \OP{C}^{n+1}_b(M)$ is the standard coboundary operator.
Like in the group case, we have the seminorm and the comparison map.
Sometimes it is convenient to work in the universal cover of $M$. Let us recall that on $\OP{C}^n_b(\wt{M})$ we have an 
action of $G = \pi_1(M)$ and $\OP{C}^n_b(M)$ is naturally isomorphic to $\OP{C}^n_b(\wt{M})^G$. 

We point out that bounded cohomology cannot be defined in terms of simplicial cochains (given some triangulation of $M$). 
For example, if the triangulation is finite, then every simplicial cochain is bounded and we get the standard cohomology. 

In this paper the manifold $M$ is aspherical. In this case  
$\OP{H}_b^n(\pi_1(M))$ is canonically isometric to $\OP{H}_b^n(M)$.
See \cite[Chapter 5]{frigerio} for a relatively elementary proof of this fact. It is based 
on an appropriate notion of resolution for bounded cohomology. 
Note that by the remarkable Mapping Theorem of Gromov, the assumption on asphericity of $M$ can be dropped,
but the proof of this fact is much harder \cite{gromov82}.

\subsection{The bounded volume class.}\label{ss:vol}

In this paper a hyperbolic manifold is a manifold (with or without boundary) whose all sectiontal curvatures equal $-1$. In particular we do not assume a hyperbolic manifold to be complete. Let $M$ be a connected, oriented, and aspherical hyperbolic $n$-manifold.
Below we recall the definition of $Vol_M$ 
and give a detailed description of its counterpart in the group cohomology $\OP{H}_b^n(\pi_1(M))$, which seems to be less known.
In Lemma \ref{l:group=top} we show that both versions give the same class under 
the canonical identification of $\OP{H}_b^n(M)$ and $\OP{H}_b^n(\pi_1(M))$.  

We start with defining the volume class in the cohomology of a group. Let $Iso_+(\mathbb{H}^n)$ denote the group of orientation preserving isometries of the hyperbolic $n$-space
and let $vol_h$ be the hyperbolic volume form on $\mathbb{H}^n$. 
Fix $* \in \mathbb{H}^n$ and for a tuple of elements $\bar{g} = (g_0,\ldots,g_n)$ in $Iso_+(\mathbb{H})$
consider the geodesic simplex $\Delta_{\bar{g}} \subset \mathbb{H}^n$ spanned by the points $g_0(*),\ldots,g_n(*)$. 
This simplex can be parametrized using the barycentric coordinates \cite[Chapter 6]{thurston}.
Therefore we regard $\Delta_{\bar{g}}$ as a map from the standard simplex to $\mathbb{H}^n$. 
We define

$$v(\bar{g}) = \int_{\Delta_{\bar{g}}} vol_h.$$

Note that $v(\bar{g})$ is the signed volume of $\Delta_{\bar{g}}$.
Moreover, $v$ is an $Iso_+(\mathbb{H}^n)$-invariant cocycle. 
Since volumes of geodesic simplices are bounded, $v$ is bounded and we can define:

$$Vol = [v] \in \OP{H}^n_b(Iso_+(\mathbb{H}^n)).$$

For $G < Iso_+(\mathbb{H}^n)$ we define $v_G$ to be the restriction of $v$ to $G$ and $Vol_G = [v_G] \in \OP{H}^n_b(G)$.

Recall that $M$ is a connected, oriented, and aspherical hyperbolic $n$-manifold. We allow $M$ to have cusps or a boundary that is not totally geodesic. 
Represent $M$ as a quotient $M = X/\pi_1(M)$ where 
$X \subset \mathbb{H}^n$ is contractible and $\pi_1(M)$ acts on $X$ by deck transformations. Note that the action of $\pi_1(M)$ on $X$ extends uniquely to an action on~$\mathbb{H}^n$.
Denote this action by $\rho \colon \pi_1(M) \to Iso_+(\mathbb{H}^n)$.
In what follows we usually do not mention the representation $\rho$ and 
regard $\pi_1(M)$ as a discrete subgroup of $Iso_+(\mathbb{H}^n)$.
To $M$ we associate the class 
$$Vol_{\rho(\pi_1(M))} \in \OP{H}_b^n(\pi_1(M)),$$ 
and write 
$Vol_M^{gp} = Vol_{\rho(\pi_1(M))}$. 
The action $\rho$ is well defined up to conjugacy, thus $Vol_M^{gp}$ does not depend on $\rho$. 

Let us now describe a singular cocycle that defines the class $Vol_M$ in $\OP{H}_b^n(M)$. 
Let $\sigma$ be a singular simplex in $X \subset \mathbb{H}^n$. 
The straightening $str(\sigma)$ is the geodesic simplex with the same vertices as $\sigma$ and parametrized 
using the barycentric coordinates (it is possible that $str(\sigma)$ is not contained in $X$).
We define 	

$$v'_M(\sigma) = \int_{str(\sigma)} vol_h.$$

We have that $v'_M$ is a $\pi_1(M)$-invariant cocycle on $X$. It defines 
the bounded class $[v'_M] \in \OP{H}_b^n(M)$ which we denote by $Vol_M$.

Hence we associate two classes to $M$: $Vol_M$ in the cohomology of the space $M$ and a class $Vol_M^{gp}$ in the group cohomology of $\pi_1(M)$. 

\begin{remark}
The definition of $Vol_M$ can be generalized by integrating, instead of $vol_h$, the pull-back $\tilde{\omega}$ of a closed $k$-form $\omega$ on $M^n$ for $k \leq n$, see \cite{bg,battista2023bounded}.
For a surface and $2$-forms such that $\tilde{\omega} = fvol_h$ for $f>0$, Theorem \ref{t:intro} holds with the same proof. As well, instead of a hyperbolic metric, one can take a pinched negatively curved metric on $M$ where one can straighten simplices. This should generate even more $Vol_M$-like classes.
\end{remark}

\begin{lemma}\label{l:group=top}
Let $M$ be a connected, oriented and aspherical hyperbolic $n$-dimensional manifold and let $r \colon \OP{H}_b^n(\pi_1(M)) \to \OP{H}_b^n(M)$ be the canonical isometric isomorphism. 
Then we have $r(Vol_M^{gp}) = Vol_M$.

\end{lemma}
\begin{proof}

Let $M = X/\pi_1(M)$ where $X \subset \mathbb{H}^n$.
First, we shall show that without loss of generality, we can assume that $X = \mathbb{H}^n$. Consider $M^{ex}= \mathbb{H}^n/\pi_1(M)$. 
Since $X~\subset~\mathbb{H}^n$, we have that $M$ is a submanifold of $M^{ex}$. Moreover, $M$ is aspherical and, by the Whitehead theorem, the inclusion $i \colon M \to M^{ex}$ is a homotopy equivalence. Hence 
$$i^* \colon \OP{H}_b^n(M^{ex}) \to \OP{H}_b^n(M)$$ 
is an isometric isomorphism. 
Moreover, $i^*(Vol_{M^{ex}}) = Vol_M$. Thus instead of $M$, we can consider $M^{ex}$. 
In other words, we can assume that $X = \mathbb{H}^n$.

The explicit formula for $r$ can be found in \cite[Lemma 5.2]{frigerio}
(note that on the standard cohomology, $r$ is just the map induced by the classifying map).
We shall give a formula for the inverse of $r$. 

Let $G = \pi_1(M)$ and $* \in \mathbb{H}^n$ be a basepoint.  
Let $\Delta$ be the map which to each tuple $\bar{g} = (g_0, \ldots, g_n) \in G^{n+1}$ 
associates $\Delta_{\bar{g}}$, 
the geodesic simplex spanned by the $g_i(*)$ and parametrized by barycentric coordinates. 

Consider the augmented cochain complexes $\{\OP{C}^\bullet_b(G),d\}$ and $\{\OP{C}_b^\bullet(\mathbb{H}^n)),d\}$.
It means that $\OP{C}^{-1}_b(G) = \OP{C}_b^{-1}(\mathbb{H}) = \mathbb{R}$ and in both cases $d_{-1}$ maps a real number to a constant function.  
The map $\Delta$ commutes with taking facets. 
That is, if $\bar{f} \subset \bar{g}$ is an $n$-tuple in an $(n+1)$-tuple $\bar{g}$, 
then $\Delta_{\bar{f}}$ is just $\Delta_{\bar{g}}$ restricted to the corresponding facet. 
Thus $\Delta$ induces a map of augmented cochain complexes 

$$\Delta^* \colon \{\OP{C}_b^\bullet(\mathbb{H}^n)),d\} \to \{\OP{C}^\bullet_b(G),d\}$$

given by $\Delta^*(c)(\bar{g}) = c(\Delta_{\bar{g}})$ and the identity on the augmentations. 
The map $\Delta^*$ is $G$-invariant and since these resolutions are relatively injective strong resolutions 
\cite[Lemma 4.12 and Lemma 5.4]{frigerio}, 
$\Delta^*$ induces a map 
$$\OP{H}^n_b(\Delta) \colon \OP{H}_b^n(M) \to \OP{H}_b^n(\pi_1(M))$$ 
which is the inverse of $r$ \cite[Theorem 4.15]{frigerio}. 
Since $r$ is an isometric isomorphism, the same holds for $\OP{H}^n_b(\Delta)$.
Moreover, it follows directly from the definitions that $\Delta^*(v'_M) = v_G$.
Thus $r(Vol_M^{gp})~=~Vol_M$. 
\end{proof}

Let $\omega_h$ be the hyperbolic volume form on a closed hyperbolic manifold $M$. 
Let us point out that $[v'_M] \in \OP{H}_b^n(M)$ goes to $[\omega_h] \in \OP{H}_{dR}(M) \simeq H^n(M)$ after applying the comparison map. 
This can be seen by the straightening of simplices homotopy, see \cite{thurston} or \cite[Lemma 8.12]{frigerio}. 

\subsection{Euler class in the mapping class group}\label{ss:euler.class}

Let $\Homeo_+^{\mathbb{Z}}(\mathbb{R})$ be the set of orientation preserving homeomorphisms $f$ of $\mathbb{R}$ that are lifts of maps from $\Homeo_+(S^1)$. That is, $f(x+1) = f(x) + 1$. It fits the central extension
$$
\mathbb{Z} \xrightarrow{} \Homeo_+^{\mathbb{Z}}(\mathbb{R}) \xrightarrow{p} \Homeo_+(S^1).
$$

The Euler class $e_b \in \OP{H}_b^2(\Homeo_+(S^1))$ is a particular bounded class representing the Euler class of this extension \cite[Chapter 10]{frigerio}. Let $S$ be a closed oriented surface of genus $\geq 2$. On $S$ we fix a hyperbolic metric. 
Let $\MCG_+(S,*)$ be the subgroup of $\MCG(S,*)$ of mapping classes represented by orientation preserving homeomorphisms. 
By the Dehn-Nielsen theorem $\MCG_+(S,*) \simeq \OP{Aut}_+(\pi_1(S,*))$, thus $\MCG_+(S,*)$ acts on the Gromov boundary $\partial \pi_1(S,*) \simeq S^1$. Hence we have a map 
$$
\alpha \colon \MCG_+(S,*) \to \Homeo_+(S^1).
$$
In geometric terms, $\alpha$ is described as follows. Let $f \in \Homeo_+(S,*)$ represent an element $\psi \in \MCG_+(S,*)$. Let $\tilde{*} \in \mathbb{H}^2$ be a fixed preimage of $*$ and let $\tilde{f}$ be the lift of $f$ such that $\tilde{f}(\tilde{*}) = \tilde{*}$. Now $\alpha(\psi)$ is the action of $\tilde{f}$ on $\partial \mathbb{H}^2$ and it does not depend on the choice of $f$ representing $\psi$.
Note that $\alpha$ restricted to $\pi_1(S,*) < \MCG_+(S,*)$ is just the standard action on $\partial \mathbb{H}^2$ by deck transformations.  

Let $e_b^{\MP} = \alpha^*(e_b)$ and $e_b^S$ be the restriction of $e_b^{\MP}$ to $\pi_1(S,*) < \MCG_+(S,*)$. The class $e_b^{\MP}$ was studied in \cite{Chen,JekelRolland}.

\section{Definitions of $\Gb$ and $\Gb^\M$}\label{s:definitions}

Suppose $M$ is a connected smooth manifold and $\omega$ is a finite measure induced by a volume form. In this section, we define the maps $\Gb$ and $\Gb^\M$ (the latter for compact $M$). We start with a geometrically motivated definition of $\Gb$ by a system of paths. This definition naturally leads to the definition of $\Gb^\M$ by a system of homeomorphisms. To show that the definitions of $\Gb$ and $\Gb^\M$ do not depend on the chosen systems, we use a result from \cite{nitsche}. To this end, we rephrase our definitions in the language of couplings. In the case of $\Gb$ the obvious coupling is given by the universal cover. For the convenience of the reader, we give the details in Section \ref{ss:pi-coupling}. It turns out that $\Gb^\M$ can be described in the language of couplings as well. However, one needs to use a bigger (and therefore disconnected) cover of $M$, see Section \ref{ss:m-coupling}.

Fix a basepoint $* \in M$. In this section, we assume that $\pi_1(M,*)$ is center free. It holds for all manifolds we are interested in and with this assumption the construction of $\Gb$ and $\Gb^\M$ is slightly simpler. 

Let $p \colon \wt{M} \to M$ be the universal cover of $M$. We view an element $l \in p^{-1}(x)$ as a homotopy class relative to $\{*,x\}$ of a path connecting $*$ to $x$. The action of $\pi_1(M,*)$ on $\wt{M}$ is given by concatenating a loop representing $\gamma \in \pi_1(M,*)$ and a path representing $l \in \wt{M}$. 

Recall that $\OP{Homeo}(M)$ is the group of all homeomorphisms of $M$ and $\OP{Homeo}_0(M)$ are those elements of $\OP{Homeo}(M)$ that are isotopic to the identity of $M$. Similarly, $\OP{Homeo}(M,\omega)$ is the group of all homeomorphims of $M$ preserving $\omega$ and $\OP{Homeo}_0(M,\omega)$ are those elements of $\OP{Homeo}(M,\omega)$ that are isotopic to the identity of $M$ via $\omega$-preserving maps. 

In all these transformation groups, we consider their right group action on $M$, so that $fg$ acts by $(fg)(x) = g(f(x))$ for all $x \in M$. We are forced to use this convention, since later in Section \ref{ss:gamma^m} we relate elements of $\pi_1(M,*)$ to homeomorphisms in $\Homeo(M)$ using a push map, and we want the multiplication in $\pi_1(M,*)$ to be compatible with how we compose homeomorphisms.  

Recall that a left action of a group $G$ on a set $X$
is a homomorphism $G \to \OP{Aut}(X)$ and a right action is an antihomomorphism $G \to \OP{Aut}(X)$. On $\OP{Aut}(X)$ we compose maps starting from the right-most. In this way we have a left action of $\pi_1(M,*)$ on $\wt{M}$ and a right action of $\OP{Homeo}(M)$ on $M$.

\subsection{Description of $\Gb$}\label{SS:gamma}

In this paragraph, we give the simplest, in our opinion, description of $\Gb$.
A more refined and general definition can be found in \cite{BM}.
Another approach based on coupling of groups is given in \cite{nitsche}. Note that the definition in \cite{nitsche} is much more general and works for measurable spaces and groups of measurable transformations.

Let $s_x$ be a path  connecting $*$ to~$x \in M$. By $[s_x]$ we denote the homotopy class of $s_x$ relative to the endpoints $\{*,x\}$ and by $\bar{s}_x$ we denote the reverse of $s_x$. A~\textbf{system of paths} for $M$ is a function on $M$ of the form $S(x) = [s_x]$ where $s_x$ is any path connecting $*$ to $x$.

Assume $S$ is a system of paths. We shall define a map 
$$ \gamma \colon \OP{Homeo}_0(M,\omega) \times M \to \pi_1(M,*).$$

Let $S(x)$ be represented by a path $s_x$. Fix $f \in \OP{Homeo}_0(M,\omega)$ and $f_t$ an isotopy connecting $Id_M$ to $f$.
We define $\gamma(f,x) \in \pi_1(M,*)$ to be the homotopy class of the loop based at $*$ which is the concatenation of $s_x$, the trajectory $f_t(x)$ and $\bar{s}_{f(x)}$.
The element $\gamma(f,x)$ does not depend on the choice of paths representing $S(x)$ and $S(f(x))$. Moreover, $\gamma(f,x)$ does not depend on the isotopy $f_t$. Indeed, it follows from \cite[Proposition 3.1]{BM} that changing the isotopy $f_t$ to another isotopy connecting $Id_M$ and $f$ would change $\gamma(f,x)$ by an element of the center of $\pi_1(M,*)$. Since we assumed that the center of $\pi_1(M,*)$ is trivial, $\gamma(f,x)$ is well defined. 
Note that $\gamma$ depends on the choice of the system of paths $S$. 

It is a simple calculation, that $\gamma$ satisfies a cocycle condition in the following form:
$$\gamma(fg,x) = \gamma(f,x)\gamma(g,f(x)).$$ 

 Define $\wt{\omega}$ to be the measure on $\wt{M}$ induced by the pull-back of the volume form on $M$ that defines $\omega$. Since elements of $\wt{M}$ are homotopy classes relative to the endpoints of paths in $M$ starting at $*$, the image $im(S)$ is a subset of $\wt{M}$.
We say that $S$ is \textbf{measurable} if $im(S)$ is $\wt{\omega}$-measurable. 
In Section \ref{ss:pi-coupling} we show that measurable systems of paths exist. 

Now assume $S$ is a measurable system of paths. For each $n$ we define 
$$\Gb \colon \OP{H}^n_b(\pi_1(M,*)) \to \OP{H}^n_b(\OP{Homeo}_0(M,\omega))$$ 
to be the map induced by the following function defined on cochains (called again~$\Gb$):

\[\Gb(c)(f_0,\ldots,f_n) = \int_M c\big(\gamma(f_0,x),\ldots,\gamma(f_n,x)\big) d\omega(x) 
.\]

In Section \ref{ss:pi-coupling} we show that the function under the integral is measurable and that~$\Gamma_b$ does not depend on the choice of a measurable system of paths.

\begin{remark}
    In \cite{BM} the group $\OP{Homeo}_0(M,\omega)$ is defined like in this paper with an additional assumption that homeomorphisms have compact support. This additional assumption was needed to define an analog of $\Gamma_b$ for the ordinary cohomology in \cite[Section 3.4]{BM}. In this paper we deal only with the bounded cohomology, therefore we do not assume that elements of $\OP{Homeo}_0(M,\omega)$ are compactly supported. Moreover, contrary to \cite{BM}, here we do not assume that $M$ carries a complete Riemannian metric. 
\end{remark}

\begin{example}
Now we explicitly describe a cocycle representing the element $\Gb(Vol_M^{gp})$. 
Let $M = X/\pi_1(M)$ where $X \subset \mathbb{H}^n$ and let $\omega$ be any finite measure induced by a volume form on $M$. Assume that $\pi_1(M)$ has trivial center (equivalently, $\pi_1(M)$ is not abelian, for example $M$ is not a quotient of $\mathbb{H}^n$ by a $\mathbb{Z}$-action). 
Let $\Tilde{*} \in X$ be a lift of the basepoint $*$ and $f_0,\ldots,f_n \in \OP{Homeo}_0(M,\omega)$. 
For every $x \in M$ we get a tuple of elements $\gamma(f_0,x), \ldots , \gamma(f_n,x)$ in $\pi_1(M)$ that act on~$X$. Consider the points $v_i(x) = \gamma(f_i,x)\Tilde{*}$ and 
span a geodesic simplex $\Delta(x)$ in~$X$ with vertices $v_i(x)$. 
Now $\Gb(Vol_M^{gp})$ is represented by a cocycle that assigns to $(f_0,\ldots,f_n)$ the average signed volume of $\Delta(x)$ over $M$ with respect to~$\omega$. 
\end{example}

\subsection{$\Gamma_b$ via coupling}\label{ss:pi-coupling}
Let us describe the construction of $\Gamma_b$ via couplings given in \cite{nitsche}. 
A coupling is a measured space $(X,\mu)$ together with a $\mu$-preserving left action of a group $\Gamma$ and a commuting $\mu$-preserving right action of a group $\Lambda$. Suppose that $\Gamma$ is countable, the action of $\Gamma$ is free and $F \subset X$ is a strict measurable fundamental domain for an action of $\Gamma$. A map $\chi \colon X \to \Gamma$ is defined by the equation $\chi(\gamma.x)=\gamma$ for $x \in F$ and $\gamma \in \Gamma$. Note that $\chi$ is measurable, i.e., the preimage of every element in $\Gamma$ is a measurable set. A function $\bar{\chi} \colon \Lambda \times F \to \Gamma$ given by $\bar{\chi}(\lambda,x) = \chi(x.\lambda)$ induces a homomorphism $tr_\Gamma^\Lambda X \colon \OP{H}_b^\bullet(\Gamma) \to \OP{H}_b^\bullet(\Lambda)$  by the formula 
\cite[Section 3]{nitsche} 

\[tr_\Gamma^\Lambda(c)(\lambda_0,\ldots,\lambda_n) = \int_F c\big(\bar{\chi}(\lambda_0,x),\ldots,\bar{\chi}(\lambda_n,x)\big)d\mu(x).\]

The map $\bar{\chi}$ is a composition of measurable functions, hence the function under the integral is measurable. Moreover, $tr_\Lambda^\Gamma X$ does not depend on the choice of $F$ \cite[Lemma 3.3]{nitsche}.

Let $f \in \Homeo_0(M,\omega)$ and suppose $f_t$ is an isotopy connecting the identity of~$M$ to $f$. Denote by $\wt{f}$ the endpoint of a lift of $f_t$ to $\wt{M}$. Note that, again by \cite[Proposition 3.1]{BM}, and by the assumption that $\pi_1(M,*)$ has a trivial center, $\wt{f}$ does not depend on the chosen isotopy $f_t$. Therefore we have a right $\wt{\omega}$-preserving action of $\Homeo_0(M,\omega)$ on $\wt{M}$, given by $x.f = \wt{f}(x)$,  $x \in \wt{M}$. Note that $x.f$ is represented by a path connecting $*$ to $x$, and then following the trajectory from $x$ to $f(x)$ given by $f_t(x)$. 

Now we consider the following coupling: $X = (\wt{M},\wt{\omega})$ together with the left action of $\Gamma = \pi_1(M,*)$ and the right action of $\Lambda = \Homeo_0(M,\omega)$. For this coupling, we have $\Gb = tr_\Gamma^\Lambda X$. Indeed, if $S$ is a measurable system of paths, $im(S)$ is a strict measurable fundamental domain of the $\pi_1(M,*)$-action, and vice versa, every measurable fundamental domain defines a measurable system of paths. Moreover, let $S$ be a measurable system of paths and $F=im(S)$. Recall that $p \colon \wt{M} \to M$ is the universal covering map. If $\gamma$ and $\bar{\chi}$ are defined by $S$ and $F$ respectively, then $\gamma(f,p(x)) = \bar{\chi}(f,x)$ for every $x \in F$ and $f \in \Homeo_0(M,\omega)$. Thus $\Gb$ and $tr_\Gamma^\Lambda X$ coincide. In particular, $\Gb$ does not depend on the choice of a measurable system of paths.

To finish the construction of $\Gb$, we shall show that measurable fundamental domains exist. Suppose $\{U_i\}_{i\in \mathbb{N}}$ is a cover of $M$ by simply connected open sets. Let $\wt{U}_i$ be a homeomorphic lift of $U_i$ to $\wt{M}$. Set $F_i = \wt{U}_i \backslash p^{-1}(\cup_{j=1}^{i-1} U_j)$. Then $F = \cup_{j=1}^\infty F_j$ is a strict measurable fundamental domain.

\subsection{Description of $\Gamma_b^\M$.}\label{ss:gamma^m}

We describe an extension of $\Gamma_b$ to a map that ranges in the bounded cohomology of $\Homeo(M,\omega)$. 
It is done by modifying the notion of a system of paths. In this and the next subsection, we assume that $M$ is compact.

Fix a point $* \in M$. Let $\Homeo(M,*)$ be the group of homeomorphisms of $M$ fixing $*$ and $\Homeo_0(M,*)$ the subgroup of homeomorphisms that are isotopic to the identity by $*$-fixing homeomorphisms. It follows from \cite[Corollary 1.1]{edwards-kirby}, that for compact $M$, $\Homeo_0(M,*)$ is the connected component of the identity in $\Homeo(M,*)$ equipped with the compact-open topology. Recall, that in these groups we compose homeomorphisms from left to right.  
Consider the mapping class group 
$$\MCG(M,*) = \Homeo(M,*)/\Homeo_0(M,*).$$ 
Let $\Homeo(M,*\to x) < \Homeo(M)$ be the subset of homeomorphisms of $M$ that send $*$ to $x$. Suppose $h_x \in \Homeo(M,*\to x)$.
Denote by $[h_x]$ the connected component of $h_x$ in $\Homeo(M,*\to x)$ with the compact-open topology. Note that two elements $h_x$ and $g_x$ in $\Homeo(M,*\to x)$ are in the same connected component if and only if there exists an isotopy $f_t$ connecting $h_x$ to $g_x$ such that $f_t(*) = x$ for all $t$. A~\textbf{system of homeomorphisms} for $M$ is a function on $M$ of the form $P(x) = [h_x]$, where $h_x \in \Homeo(M,*\to x)$.
For example, it can be a point-pushing map along~$s_x$, where $S(x) = [s_x]$ is a system of paths. 

Let $P(x) = [h_x]$ be a system of homeomorphisms. For every $f \in \Homeo(M,\omega)$, we have $h_x \circ f \circ h_{f(x)}^{-1} \in \Homeo(M,*)$. We define a cocycle 
$$
\gamma^\M \colon \Homeo(M,\omega) \times M \to \MCG(M,*)
$$
by $\gamma^\M(f,x) = [h_x \circ f \circ h_{f(x)}^{-1}]$. Note that $\gamma^\M$ does not depend on the homeomorphisms $h_x$ representing $P(x)$ but depends on the choice of $P$. There is a notion of a measurable system of homeomorphisms, we discuss it in Section \ref{ss:m-coupling}.

Likewise in Subection~\ref{SS:gamma}, any measurable system of homeomorphisms $P$ induces a map 
$$
\Gamma_b^\M \colon \OP{H}_b^n(\MCG(M,*)) \to \OP{H}_b^n(\Homeo(M,\omega)).
$$

In Section \ref{ss:m-coupling} we show that $\Gb^\M$ does not depend on the choice of $P$. 
Note that $\Gamma_b^\M$  generalizes the homomorphism $\mathcal{G}_{S,1}$ from \cite{BMentropy} defined only for surfaces $S$ and ranging in quasimorphisms on $\Homeo(S,\omega)$.

\begin{remark}
If $M$ is oriented, we can consider the oriented version of $\Gamma_b^\M$.
Namely, let $\MCG_+(M,*)$ be the subgroup of $\MCG(M,*)$ of mapping classes represented by orientation preserving homeomorphisms and let $\Homeo_+(M,\omega)$ be homeomorphisms preserving $\omega$ and the orientation. Assuming that for every $x\in M$, $P(x)$ is represented by an orientation preserving homeomorphism, we have $\gamma^\M(f,x) \in \MCG_+(M,*)$ for $f \in \Homeo_+(M,\omega)$. Thus we can define 
$$
\Gamma_b^\MP \colon \OP{H}_b^n(\MCG_+(M,*)) \to \OP{H}_b^n(\Homeo_+(M,\omega)).
$$
This map is used in Theorem \ref{t:intro2}.
\end{remark}

Let us now explain in what sense $\Gamma_b^\M$ extends $\Gamma_b$.
Note that $\MCG(M,*) = \pi_0(\Homeo(M,*))$ and $\MCG(M) = \pi_0(\Homeo(M))$. 
Consider the fiber bundle

$$\Homeo(M,*) \xrightarrow{} \Homeo(M) \xrightarrow{ev} M,$$

where $ev(f) = f(*)$. The long exact sequence of the homotopy groups gives
$$
\pi_1(\Homeo(M)) \xrightarrow{ev_1} \pi_1(M,*) \xrightarrow{Pu} \MCG(M,*) \xrightarrow{F} \MCG(M).
$$ 

For $g \in \pi_1(M,*)$, $Pu(g)$ is the mapping class represented by the time-one map $f_1$ of a point-pushing isotopy $\{f_t\}_{t\in[0,1]}$ along a loop representing $g$. By \cite[Proposition 3.1]{BM}, $im(ev_1)$ lies in the center of $\pi_1(M,*)$, therefore $ev_1$ is trivial and $Pu$ is an embedding. 

Let $[f] \in im(Pu) = ker(F) < \MCG(M,*)$ and let $f_t$ be an isotopy connecting the identity to $f$ in $\Homeo_0(M)$. Let $Tr([f])$ be the element of $\pi_1(M,*)$ represented by $f_t(*)$. By the triviality of $ev_1$, this map is well defined, and it is the inverse of $Pu$ on $im(Pu)$. 

Let $S$ be any measurable system of paths and $P$ a system of homeomorphims such that $P(x) = [h_x]$ with $h_x$ the point-pushing map along a path $s_x$ such that $S(x) = [s_x]$. The cocycles $\gamma$ and $\gamma^\M$ are defined with respect to these systems. For every $f \in \Homeo_0(M,\omega)$ we have $\gamma(f,x) = Tr(\gamma^\M(f,x))$.
Thus $Pu(\gamma(f,x)) = \gamma^\M(f,x)$ and we have a commutative diagram:

$$
\begin{tikzcd}
\Homeo(M,\omega) \times M  \arrow{r}{\gamma^\M}& \MCG(M,*) \\
\Homeo_0(M,\omega) \times M \arrow{u}\arrow{r}{\gamma}& \pi_1(M,*) \arrow{u}{Pu}.
\end{tikzcd}
$$

Thus the following diagram commute:

$$
\begin{tikzcd}
\OP{H}_b^\bullet(\MCG(M,*)) \arrow{d}{Pu^*}\arrow{r}{\Gamma_b^\M}& \OP{H}_b^\bullet(\Homeo(M,\omega)) \arrow{d}\\
\OP{H}_b^\bullet(\pi_1(M,*)) \arrow{r}{\Gamma_b}& \OP{H}_b^\bullet(\Homeo_0(M,\omega)).
\end{tikzcd}
$$

\subsection{$\Gb^\M$ via couplings}\label{ss:m-coupling}

We shall construct a cover of $M$ on which $\MCG(M,*)$ acts on the left and $\Homeo(M,\omega)$ acts on the right.

Recall that $ev \colon \Homeo(M) \to M$ is a fibre bundle defined by $ev(f) = f(*)$.
Note that $\Homeo(M,* \to x) = ev^{-1}(x)$.
Denote by $\wt{M}^\M$ the set of connected components of the fibers of~$ev$, and by $q \colon \Homeo(M) \to \wt{M}^\M$ the quotient map. On $\wt{M}^\M$ we consider the quotient topology. Note that we have $q(f) = q(g)$ if and only if $f$ can be connected to $g$ via an isotopy
preserving $f(*)=g(*)$ at all times.

The map $ev$ factors via $q$:

$$
\begin{tikzcd}
\Homeo(M) \arrow{d}{q}\arrow{r}{ev}& M\\
\wt{M}^\M \arrow{ur}{p^\M}&
\end{tikzcd}
$$

where $p^\M$ is the unique map making the above diagram commutative. 

The evaluation map is a fiber bundle, thus every $x \in M$ has an open neighborhood $U$ such that $ev^{-1}(U)$ is homeomorphic to $ U \times \Homeo(M,*)$ by a fiber preserving homeomorphism. Moreover, by the definition of $q$, we have that $q(ev^{-1}(U))$ is the disjoint union of copies of $U$. On each such a copy $p^\M$ is a homeomorphism. Thus $p^\M$ is a covering map. 

Denote by $\wt{\omega}^\M$ the measure on $\wt{M}^\M$ defined by the pull-back of the volume form on $M$ that defines $\omega$. We shall define the mentioned left and right actions. 

Let $[c] \in \wt{M}^\M$ be a connected component of $ev^{-1}(c(*))$ represented by $c \in \Homeo(M)$, and $[h] \in \MCG(M,*)$ be a mapping class represented by $h \in \Homeo(M,*)$. The left action of $\MCG(M,*)$ on $\wt{M}^\M$ is given by $[h].[c] = [hc] \in \wt{M}^\M$ (note that we compose homeomorphisms starting from the left), where $[hc]$ denotes the connected component of $ev^{-1}(c(*))$ containing $hc$. This action is transitive on the fibers of $p^\M$ and permutes the components of $q(ev^{-1}(U))$, where $U$ is as above. Therefore it is a deck-transformation group of $
\wt{M}^\M$ and preserves $\wt{\omega}^\M$. 

Let $[c] \in \wt{M}^\M$ and $f \in \Homeo(M,\omega)$. The right action of $\Homeo(M,\omega)$ on $\wt{M}^\M$ is given by $[c].f = [cf]$, where $[cf]$ is the connected component of $ev^{-1}(f(c(*))$ containing $cf$. This action covers the $\Homeo(M,\omega)$-action on $M$ and thus is $\wt{\omega}^\M$-preserving. 

The above actions define a coupling on $\wt{M}^\M$. Moreover, by construction, if $P$ is a system of homeomorphisms, $im(P)$ is a subset of $\wt{M}^\M$. We say that a \textbf{system of homeomorphisms $P$ is measurable} if $im(P)$ is $\wt{\omega}^\M$-measurable.

Every measurable system of homeomorphisms defines a strict fundamental domain for the $\MCG(M,*)$-action and vice versa. It follows directly from definitions (see Section \ref{ss:pi-coupling}) that the constructions described in Section \ref{ss:gamma^m} and \cite{nitsche} coincide. In particular, $\Gb^\M$ does not depend on the choice of a measurable system of paths \cite[Lemma 3.3]{nitsche}. 

Note that $\wt{M}^\M$ is disconnected and contains the universal cover of $M$. Indeed, 
let $\wt{M}_0$ be the subset of $\wt{M}^\M$ containing classes that are represented by homeomorphisms isotopic to the identity in $\Homeo(M)$. The subgroup $\pi_1(M,*) < \MCG(M,*)$ acts on $\wt{M}_0$ and the quotient is $M$, thus $\wt{M}_0$ is the universal cover of $M$. An explicit isomorphism between $\wt{M}_0$ and $\wt{M}$ is given by the map $T \colon \wt{M}_0 \to \wt{M}$ where $T([c])$ is the homotopy class of the path $c_t(*)$ traced by any isotopy $c_t$ connecting~$Id_M$ to~$c$. By the triviality of $ev_1 \colon \pi_1(\Homeo(M)) \to \pi_1(M,*)$, this map is well defined. Therefore, $\wt{M}^\M$ consists of infinitely many copies of $\wt{M}$ indexed by the right cosets of $\pi_1(M,*)$ in $\MCG(M,*)$.

Finally, we show that measurable systems of homeomorphisms exist. It follows from the existence of measurable systems of paths. Let $S(x) = [s_x]$ be a measurable system of paths and let $P(x) = [h_x]$, where $h_x$ is a point-pushing map along~$s_x$. Using the isomorphism $T \colon \wt{M}_0 \to \wt{M}$ we can regard $im(S)$ as a $\wt{\omega}^\M$-measurable subset of $\wt{M}_0$, and under this identification $im(S) = im(P)$. Thus $P$ is a measurable system of homeomorphisms.   

\begin{remark}
    Set $\pi = \pi_1(M,*)$. The cover $\wt{M}^\M$ is isomorphic to $(\MCG(M,*) \times \wt{M})/_\pi$, where $\pi$ acts on $\MCG(M,*) \times \wt{M}$ by $\gamma.(h,x) = (h\gamma^{-1},\gamma.x)$.
\end{remark}

\section{Restriction to a free subgroup}\label{s:restricting}

In this section, we find a free subgroup $F$ of $\pi_1(M,*)$ such that volume and Euler classes restricted to $F$ have positive norms. 

\subsection{Volume class in dimension 2}\label{ss:dim2}
Let $X$ be a topological space. The $l^1$-homology of $X$ is denoted by $\OP{H}^{l_1}_n(X)$ \cite[Chapter 6]{frigerio}.
In $l^1$-homology we allow chains to be infinite sums $c = \Sigma_i^\infty a_i\sigma_i$,
where $\n c \n = \Sigma_i^\infty |a_i| < \infty$.
As usual, the norm on chains induces the norm on $\OP{H}^{l_1}_n(X)$. 
We have a Kronecker product between $l^1$-homology and bounded cohomology:
$$\langle \cdot , \cdot \rangle \colon \OP{H}_b^n(X) \times \OP{H}_n^{l_1}(X) \to \B R.$$
The Kronecker product is defined on the level of chains by 
$\langle b, a \rangle = \Sigma_i^\infty a_ib(\sigma_i)$ 
where $a = \Sigma_i^\infty a_i\sigma_i$ and $b$ is a bounded cochain. 
Moreover we have $| \langle B, A \rangle | \leq \n B \n \n A \n$ where 
$B \in \OP{H}_b^n(X)$ and $A \in \OP{H}_n^{l_1}(X)$.
The following lemma is a variation of a result obtained in \cite{mitsumatsu}. We do not assume that $S$ is closed.  

\begin{lemma}\label{l:Vol_S}
Let $S$ be an oriented hyperbolic surface with non-abelian fundamental group. 
Then $Vol_S$ has a positive norm. 
\end{lemma}
\begin{proof}
We shall find 
$C\in\OP{H}_2^{l_1}(S)$ such that $\langle Vol_S, C\rangle\neq 0$.
We can assume that $S$ is a quotient of $\mathbb{H}^2$, as in the beginning of the proof of Lemma \ref{l:group=top}. 
Let $p\colon \mathbb{H}^2\to S$ be the covering map and let $G = \pi_1(S)$.
Since $G$ is a surface group or is free, the commutator subgroup $[G,G]$ is non-abelian and hence
contains a hyperbolic element $\gamma \in [G,G]$ \cite[Theorem 2.4.4]{skatok}. 
Let $A_\gamma \subset \mathbb{H}^2$ be the axis of $\gamma$.
The conjugacy class of $\gamma$ is represented by the closed geodesic $L_\gamma = p(A_\gamma)\subset S$. We regard $L_\gamma$ as a map from $[0,1]$ to $S$.
Since $\gamma \in [G,G]$, $\gamma$ is homologically trivial. 
Hence there exists a triangulated subsurface $S_0 \subset S$ such that $\partial S_0 = L_\gamma$.

The loop $L_\gamma$ is the boundary of an $l^1$-chain $c_0$ whose simplices are contained in the image of $L_\gamma$ \cite[Section 3]{mitsumatsu}. 
Thus $c = S_0-c_0$ is an $l^1$-cycle. 
Recall that $Vol_S = [v'_S]$.
The straightening of every simplex in $c_0$ is degenerate ($L_\gamma$ is a geodesic), thus we have $\langle v'_S, c_0 \rangle = 0$. 
We can assume that the triangulation of $S_0$ consists of geodesic simplices, hence $\langle v'_S, S_0\rangle$ equals the hyperbolic volume of $S_0$. Hence if we set $C~=~[c]$, we obtain $\langle Vol_S, C \rangle > 0$.
The inequality $\langle Vol_S, C \rangle \leq \n Vol_S \n \n C \n $
implies that the norm of $Vol_S$ is positive. 
\end{proof}

The following corollary is stated in the group 
version terms of the volume class. 

\begin{corollary}\label{c:2free}
Let $S$ be an oriented hyperbolic surface with non-abelian fundamental group. There exists an embedding $i \colon F \to \pi_1(S)$ of a free non-abelian group~$F$ such that $i^*(Vol_S^{gp})$ has a positive norm.
\end{corollary}
\begin{proof}
    Let $F$ be any free non-abelian subgroup of $S$ (which can be equal to $\pi_1(S)$ if $S$ is not closed).
    We have $i^*(Vol_S^{gp}) = Vol_{i(F)}$. 
    By Lemma \ref{l:group=top} we know that the norm of $Vol_{i(F)}$ is 
    equal to the norm of $Vol_{S'}$ where $S' = \mathbb{H}^2/i(F)$
    and by Lemma \ref{l:Vol_S} the norm of $Vol_{S'}$ is positive. 
\end{proof}

\subsection{Volume class in dimension 3}\label{ss:dim3}

For some Kleinian groups, i.e. the discrete subgroups of $Iso_+(\mathbb{H}^3)$, 
the volume class was studied in \cite{soma}.
Note that if $G$ is a torsion-free Kleinian group, 
then it acts freely and properly discontinuously on $\mathbb{H}^3$. 
Thus $\mathbb{H}^3/G$ is a manifold.

\begin{theorem}{\cite[Theorem 1]{soma}}\label{l:soma}
Let $G < Iso_+(\mathbb{H}^3)$ be a torsion-free topologically tame Kleinian group such that the volume of $M = \mathbb{H}^3/G$ is infinite. 
Then $Vol_M$ has a positive norm if and only if $G$ is not elementary and geometrically infinite. 
\end{theorem}

A Kleinian group $G$ is geometrically finite if $N_\epsilon(H(L_G)/G)$ has finite volume for some $\epsilon>0$, 
where $N_\epsilon$ is an $\epsilon$ neighborhood
and $H(L_G)$ is the convex closure of the limit set of $G$ \cite[Chapter 8, Definition 8.4.1]{thurston}. A torsion-free Kleinian group $G$ is topologically tame if $\mathbb{H}^3/G$ is homeomorphic to the interior of a compact manifold. 

Note that every discrete finitely generated non-abelian free subgroup $F$ of $Iso_+(\mathbb{H}^3)$ is topologically tame \cite{agol2004tameness, calegari-gabai}, not elementary and $\mathbb{H}^3/F$ has infinite volume (otherwise, by the thick-thin decomposition, $\mathbb{H}^3/F$ would have a cusp and $\mathbb{Z}^2$ would embed in $F$). 
Let $M$ be an oriented $3$-dimensional hyperbolic manifold. 
Then for every free group $F$ in $\pi_1(M) < Iso_+(\mathbb{H}^3)$ which is geometrically infinite, $Vol_F$ has positive norm. 
Below we describe our main example of such a situation, i.e., manifolds that fiber over the circle with non-compact fiber.

\begin{example}\label{e:fiber}
Suppose $S$ is a connected oriented surface without boundary and free fundamental group $F = \pi_1(S)$. 
Let $f \in \OP{Diff}^{+}(S)$. 
The mapping torus of $f$ is a $3$-dimensional manifold 
$M_f = S \times [0,1] / \sim$ where $(x,0) \sim (f(x),1)$.
That is, the boundary components of $S \times [0,1]$ are glued together via $f$. 
Note that $M_f$ fibers over the circle with fiber $S$, and conversely, every $3$-manifold 
that fibers over a circle with fiber $S$ can be constructed in this way. 
The mapping torus $M_f$ is hyperbolic
if and only if $f$ is isotopic to a pseudo-Anosov map \cite{thurstonII}.
If $M_f$ is hyperbolic, the hyperbolic structure is unique and we have a unique class $Vol^{gp}_{M_f} \in \OP{H}_b^3(\pi_1(M_f))$.  
Now the inclusion of the fiber $S$ into $M_f$
gives an embedding $i \colon F \to \pi_1(M_f)$ and $i(F)$ is a normal subgroup of $\pi_1(M)$. It follows that the limit set of $i(F)$ is equal to the limit set of $\pi_1(M)$ \cite[Chapter 8, Corollary 8.1.3]{thurston}, thus it is the entire sphere at infinity. Hence $i(F)$ is geometrically infinite.
By Theorem \ref{l:soma} we have that $i^*(Vol^{gp}_{M_f}) = Vol_{i(F)}$ has positive norm. 
\end{example}

\subsection{Euler class}\label{ss:euler}
We briefly recall basic definitions concerning quasimorphisms \cite[Chapter 2]{frigerio}. Let $G$ be a group.  A real function $q \colon G \to \mathbb{R}$ is called a quasimorphism if there exists some $D \in \mathbb{R}$ such that 
$$|q(ab)-q(a)-q(b)| \leq D$$ for any $a,b \in G$. The minimal such $D$ is called the defect of $q$. A quasimorphism is homogeneous if $q(a^n) = nq(a)$ for any $a\in G$ and $n \in \mathbb{Z}$. The non-homogeneous coboundary $dq(a,b) = q(a)-q(ab)+q(b)$ of $q$ is interpreted as a second bounded cohomology class $[dq] \in \OP{H}_b^2(G)$. 
If $q$ is homogeneous and not a homomorphism, then $[dq]$ is non-trivial and has a positive norm \cite[Corollary 6.7]{frigerio}. 

\begin{lemma}\label{l:euler_restriction}
Let $S$ be an oriented closed surface of genus $\geq 2$.
There exists an embedding $i \colon F \to \pi_1(S)$ of a free non-abelian group $F$ such that $i^*(e_b^S)$ has a positive norm. 
\end{lemma}

\begin{proof}
Let $T^1S$ be the unit tangent bundle of $S$ and denote by $q \colon \pi_1(T^1S) \to \pi_1(S)$ the map induced by the projection $T^1S \to S$. 
We will use the rotation quasimorphism $Rot \colon \pi_1(T^1S) \to \mathbb{R}$ defined in \cite{Huber}.
It is a homogeneous quasimorphism of defect $1$ which trivialises the pull-back $q^*(e_b^S) \in \OP{H}_b^2(\pi_1(T^1S))$, i.e., $q^*(e_b^S) = [d Rot]$ \cite[Theorem 5.9]{Huber}. 

Let $a,b \in \pi_1(T^1S)$ be such that $Rot(ab) \neq Rot(a)+Rot(b)$. Let $F = \la q(a),q(b) \ra < \pi_1(S)$. Denote this inclusion by $i \colon F \to \pi_1(S)$ and set $F' = q^{-1}(F)$. It follows from the definition of $a$ and $b$, that $Rot_{|F'}$ is a homogeneous quasimorphism that is not a homomorphism. Thus $q^*(e^S_b)_{|F'}$ has positive norm and $i^*(e^S_b)$ must have positive norm, since $q^*i^*(e^S_b) = q^*(e^S_b)_{|F'}$. Finally, $F$ must be free of rank $2$. Indeed, every subgroup of $\pi_1(S)$ is free non-abelian, abelian, or is a surface group. But surface groups are not generated by $2$ elements and abelian groups do not carry a non-trivial class in their second bounded cohomology. Thus $F$ is free non-abelian of rank $2$.
\end{proof}

\section{Proof of the theorem}\label{s:proof}
Let $M$ be a manifold with a volume form $\omega$.
As usual, $\omega$ denotes as well the induced measure and we assume that this measure is finite. 
Suppose $i \colon F \to \pi_1(M)$ is an embedding and consider $i^* \colon \OP{H}_b^\bullet(\pi_1(M)) \to \OP{H}_b^\bullet(F)$. 
Let $\rho \colon F \to \OP{Homeo}_0(M,\omega)$ be a representation of $F$ by homeomorphisms. 
Let $\rho^* \colon \OP{H}_b^\bullet(\OP{Homeo}_0(M,\omega)) \to \OP{H}_b^\bullet(F)$.
Thus we have a not necessarily commutative diagram:
$$
\begin{tikzcd}
	\OP{H}_b^\bullet(\pi_1(M)) \arrow[r, "\Gamma_b"] \arrow[d,"i^*"] &  
	\OP{H}_b^\bullet(\OP{Homeo}_0(M,\omega)) \arrow[dl,"\rho^*"]\\
	\OP{H}_b^\bullet(F) 
\end{tikzcd}
$$
Let $\Lambda, \epsilon \in \mathbb{R}$. We say that $\rho$ is an $(F, \Lambda,\epsilon)$-inverse of $\Gamma_b$ if for every $C \in \OP{H}_b^\bullet(\pi_1(M))$ we have
$$\n \rho^*\Gamma_b (C)-\Lambda i^*(C) \n \leq \epsilon \n C \n.$$

\begin{lemma}\label{l:representation}
    Let $M$ be a manifold and $\omega$ a finite measure induced by a volume form on~$M$. 
    Suppose that $i \colon F \to \pi_1(M)$ is an embedding of a non-abelian free group $F$. There exists $\Lambda \in \mathbb{R}$ such that  
    for every $\epsilon>0$ there exists an $(F, \Lambda,\epsilon)$-inverse of $\Gamma_b$.
\end{lemma}

\begin{proof} 
Let $\OP{dim}(M)=m$. Denote by $B^{m-1} \subset \mathbb{R}^{m-1}$ the $m-1$ dimensional closed unit ball, and let $S^1 = \mathbb{R}/\mathbb{Z}$. 
Let us fix $\eta \in (0,1)$ and define an isotopy $P^t_\eta~\in~\OP{Diff}(S^1 \times B^{m-1})$ by  
$$P^t_\eta(\psi,x)=(\psi+tf(\n x\n),x) \hspace{0.4cm} \forall (\psi,x) \in S^1 \times B^{m-1},$$ 
where $t \in [0,1]$, and $f:[0,1]\to \mathbb{R}$ is a~smooth function such that 
$f(y) = 1$ for $y \leq 1-\eta$ and $f(1) = 0$.
We call $P^t_\eta$ the finger-pushing isotopy and $P^1_\eta$ the finger-pushing map. 
Note that $P^0_\eta = Id$ and that $P^1_\eta$ fixes point-wise the boundary of $S^1 \times B^{m-1}$
and fixes all points $(\psi,x)$ for which $\n x \n \leq 1-\eta$. Moreover, $P^t_\eta$ fixes the boundary of $S^1 \times B^{m-1}$ for all $t$. 
Denote by $g_0$ be the product of the standard Euclidean Riemannian metrics on $B^{m-1}$ and $S^1$.
By the theorem of Fubini, the measure induced by $g_0$~is preserved 
by the map $P^t_\eta$ for every $t \in [0,1]$.
Let $a_1,\ldots, a_k$ be generators of $F:=F_k$, where $k>1$. 
We represent $i(a_i)$ by a loop $\alpha_i$ 
which is based at $* \in M$. 

Let $B$ be a closed ball in $M$ containing $*$ and set $\Lambda = \omega(B)$. Suppose $A_i$ are closed small tubular neighborhoods of $\alpha_i$.
Then $N_i=B\cup A_i$~is a~closed neighborhood of $\alpha_i$ which is diffeomorphic to $S^1 \times B^{m-1}$.
Let $P^t_{\eta}(\alpha_i) \in \Diff_0(M)$~be the isotopy
defined by pulling-back $P^t_{\eta}$ via a diffeomorphism $n_{\alpha_i} \colon N_i \to S^1 \times B^{m-1}$ (i.e., we have $P^t_{\eta}(\alpha_i) = n_{\alpha_i}^{-1} \circ P^t_{\eta} \circ n_{\alpha_i}$ on $N_i$), and extending it by the identity outside~$N_i$. Note that the Moser trick \cite{Moser} allows us to choose $n_{\alpha_i}$ such that $P^t_{\eta}(\alpha_i)$ preserves~$\omega$.
Let $S_i$ be the support of $P^1_{\eta}(\alpha_i)$. It is a small thickening of the boundary of $N_i$.

The homomorphism $\rho \colon F \to \OP{Homeo}_0(M,\omega)$  is given~by: 
$$\rho(a_i) = P^1_{\eta}(\alpha_i).$$
To simplify the notation, we identify $F$ with its image $i(F)$.
Now we consider the values of $\gamma$ on elements of the form $(\rho(w),x)$, 
where $w\in F, x \in M$.

From the description of $\gamma$ in Section \ref{SS:gamma} we have:
$$
\gamma(\rho(w),x) = 
\begin{cases}
	e & x \in M-\bigcup_{i=1}^kN_i,\\
	w & x \in B - \bigcup_{i=1}^k S_i,\\
	? & x \in (\bigcup_{i=1}^k A_i - B) \cup \bigcup_{i=1}^k S_i.
\end{cases}
$$
Let $C \in \OP{H}_b^n(M)$ and let $c$ be a bounded cochain representing $C$. Since any cochain and its anti-symmetrisation define the same class, we may assume that $c(e,\ldots,e)~=~0$. 
Let $$\bar{f}~=~(f_0,f_1,\ldots, f_n)\in~\OP{Homeo}_0(M,\omega)^{n+1},$$ and denote 
$$\gamma(\bar{f},x)=(\gamma(f_0,x),\gamma(f_1,x),\ldots,\gamma(f_n,x)).$$ 
Let $\overline{w} \in F^{n+1}$. We have: 
$$\rho^*\Gamma_b (c)(\overline{w})=\Gamma_b (c)(\rho(\overline{w}))=\int_{M} c(\gamma(\rho(\overline{w}),x))d\omega(x).$$
Denote $E :=(\bigcup_{i=1}^k A_i - B) \cup \bigcup_{i=1}^k S_i$. We obtain

\begin{align*}\label{e:split} 
	\rho^*\Gamma_b (c)(\overline{w}) &= 
     \int_{B - \bigcup_{i=1}^k S_i}c(\overline{w})d\omega(x) +
	\int_{E}c(\gamma(\rho(\overline{w}),x))d\omega(x)\\
    &=\omega\left(B - \bigcup_{i=1}^k S_i\right)i^*(c)(\overline{w})+
    \int_{E}c(\gamma(\rho(\overline{w}),x))d\omega(x).
\end{align*}
Let 
$$c_{res}(\overline{w}):=\int_{E}c(\gamma(\rho(\overline{w}),x))d\omega(x).$$
Note that $c_{res}$ represents a class in $\OP{H}_b^n(F)$ and we can write
$$
\rho^*\Gamma_b (c) = 
\omega(B - \bigcup_{i=1}^k S_i)i^*(c) + c_{res},
$$
and  
$$\n c_{res}\n \leq \omega(E)\n c\n.$$
Moreover:
\begin{align*}
&\n \rho^*\Gamma_b (c)-\omega(B)i^*(c)\n\leq\\ 
&\leq \omega(\bigcup_{i=1}^k S_i)\n i^*(c)\n +\omega(E)\n c\n \leq\\
& \leq \big[ \omega (\bigcup_{i=1}^k S_i) + \omega(E) \big] \n c \n\ .
\end{align*}
Now $\omega(S_i)$ and $\omega(E)$ can be taken to be arbitrarily small by taking small $\eta$ and small neighborhoods $A_i$. The cochain $c$ was any cochain representing $C$. Thus for any chosen $\epsilon$ we can have:
$$\n \rho^*\Gamma_b (C)-\Lambda i^*(C)\n\leq \epsilon \n C \n,$$
where $\Lambda = \omega(B)$.
\end{proof}

\begin{theorem}\label{t:main}
    Let $M$ be an oriented manifold of dimension $n$ such that it is either:
    \begin{itemize}
        \item A hyperbolic surface with a non-abelian fundamental group or
        \item A complete hyperbolic $3$-manifold whose fundamental group contains a geometrically infinite finitely generated free group (e.g. $M$ fibers over the circle with non-compact fiber).
    \end{itemize}
    Let $\omega$ be a volume form on $M$ such that the induced measure is finite. Then $\Gamma_b(Vol_M^{gp}) \in \OP{H}_b^n(\OP{Homeo}_0(M,\omega))$ has positive norm.
\end{theorem}

\begin{proof}
    By Corollary \ref{c:2free} and Theorem \ref{l:soma} in both cases we have a free group $F$ and 
    an embedding $i \colon F \to \pi_1(M)$ such that $i^*(Vol_M^{gp})$ has positive norm. 
    Let $\rho \colon F \to \OP{Homeo}_0(M,\omega_h)$ be an $(F, \Lambda,\epsilon)$-inverse of $\Gamma_b$ with 
    $\epsilon$ satisfying 
    $$0< \Lambda \n i^*(Vol_M^{gp})\n - \epsilon \n Vol_M^{gp}\n.$$ 
    We have 
    $$\Lambda \n i^*(Vol_M^{gp})\n - \n\rho^*\Gamma_b (Vol_M^{gp})\n 
    \leq \n\Lambda i^*(Vol_M^{gp}) - \rho^*\Gamma_b (Vol_M^{gp})\n 
    \leq \epsilon \n Vol_M^{gp}\n .$$
    Thus 
    $$0 < \Lambda \n i^*(Vol_M^{gp})\n - \epsilon \n Vol_M^{gp}\n \leq \n\rho^*\Gamma_b (Vol_M^{gp})\n.$$
    Since $\rho^*$ is a contraction, $\Gamma_b (Vol_M^{gp})$ must have positive norm.
\end{proof}

\begin{theorem}\label{t:main2}
Let $S$ be an oriented closed surface of genus $\geq 2$ and $\omega$ a measure induced by an area form on~$S$.
Then the classes $\Gamma_b(e_b^S) \in \OP{H}_b^2(\Homeo_0(S,\omega))$
and $\Gamma_b^{\MP}(e_b^{\MP}) \in \OP{H}_b^2(\Homeo(S,\omega))$ have positive norms.
\end{theorem}

\begin{proof}
The proof for $e_b^S$ is the same as in Theorem \ref{t:main} using Lemma \ref{l:euler_restriction}.
Note that elements of $\Homeo(S,\omega)$ automatically preserve the orientation of $S$. 
Positivity of the norm of $e_b^{\MP}$ follows from the commutative diagram

$$
\begin{tikzcd}
\OP{H}_b^2(\MCG_+(S,*)) \arrow{d}{Pu^*}\arrow{r}{\Gamma_b^{\MP}}& \OP{H}_b^2(\Homeo(S,\omega)) \arrow{d}\\
\OP{H}_b^2(\pi_1(S,*)) \arrow{r}{\Gamma_b}& \OP{H}_b^2(\Homeo_0(S,\omega)),
\end{tikzcd}
$$

where $Pu \colon \pi_1(S,*) \to \MCG_+(S,*)$ is the injection from the Birman exact sequence. 
\end{proof}

\section{Dirac measure}\label{s:dirac}

The constructions of $\Gb$ and $\Gb^\M$ are flexible and admit more variants. First of all, one does not need to restrict to measures coming from a volume form. What is needed, is a measure with a cocycle for which the integral in the definition is well-defined. Moreover, one can relax the definition of an isotopy. For example, it is not necessary to assume that isotopy preserves the measure at all times. Isotopy might be as well substituted by homotopy. 

In this short section, we discuss the (somewhat degenerate) case where the measure is the Dirac measure and isotopies do not preserve the measure. In this case, $\Gb$ is induced by a homomorphism.
Let $M$ be a manifold and $* \in M$ a basepoint. 
We assume that the center of $\pi_1(M,*)$ is trivial (what we really need to assume is the triviality of $ev_1$, and even in the case where it is not, one could substitute $\pi_1(M,*)$ with the quotient $\pi_1(M,*)/im(ev_1)$). By $*$ we denote as well the Dirac measure centered on $*$.

Let $G$ be the subgroup of $\Homeo_0(M)$ of all homeomorphisms $f$ preserving $*$. Thus an element of $G$ is isotopic to the identity by an isotopy that can move $*$. Suppose $S$ is a system of paths. As in Section \ref{SS:gamma}, we get a cocycle:

$$
\gamma \colon G \times M \to \pi_1(M,*)
$$

and a map
$$\Gamma_b \colon \OP{H}_b^\bullet(\pi_1(M,*))\to\OP{H}_b^\bullet(G).$$

Note that on $\wt{M}$ we can consider the counting measure on the orbit $p^{-1}(*)$. With such a measure $\wt{M}$ defines a coupling and every $S$ is measurable. Moreover, $\Gamma_b$ does not depend on $S$ \cite[Lemma 3.3]{nitsche}.

Recall that we have a homomorphism 

$$Tr \colon G \to \pi_1(M,*)$$

defined in the following way: $Tr(f)$ is the homotopy class of the loop $f_t(*)$, where $f_t$ is any isotopy between $Id_M$ and $f$. 

It is straightforward to see that $\Gamma_b = Tr^*$, the map induced on bounded cohomology by $Tr$. Note that if we started with the group $\Homeo_0(M,*)$ (isotopies preserve $*$ at all times), instead of $G$, then 
$\Gamma_b \colon \OP{H}_b^\bullet(\pi_1(M,*)) \to \OP{H}_b^\bullet(\Homeo_0(M,*))$ 
would be trivial in positive degrees. Indeed, in this case $\Gamma_b(c)$ is a constant cocycle, and constant cocycles in positive degrees represent trivial classes.

Suppose that a non-abelian free group $F$ embeds in $\pi_1(M,*)$. The representations $\rho$ constructed in Lemma \ref{l:representation} are homomorphisms. Since $\gamma(\rho(w),*) = w$ for every $w \in F$, the following diagram commutes:
$$
\begin{tikzcd}
	\OP{H}_b^\bullet(\pi_1(M,*)) \arrow[r, "\Gamma_b"] \arrow[d,"i^*"] &  
	\OP{H}^\bullet_{b}(G) \arrow[dl,"\rho^*"]\\
	\OP{H}^\bullet_b(F) 
\end{tikzcd}
$$
Thus $\rho$ is a $(F,1,0)$-inverse of $\Gb$. 
It follows that Theorem \ref{t:main} and Theorem \ref{t:main2} hold as well for the Dirac measure. Similarly, Theorem A and Theorem B from \cite{BM} hold with $\mathcal{T}_M = G$, i.e., $Tr^*$ has the image of dimension continuum in degree $2$ and $3$ (if $M$ satisfies the assumptions of Theorems A and B). 

\bibliographystyle{alpha}
\bibliography{bibliography}

\newcommand{\etalchar}[1]{$^{#1}$}
\def\polhk#1{\setbox0=\hbox{#1}{\ooalign{\hidewidth
  \lower1.5ex\hbox{`}\hidewidth\crcr\unhbox0}}}
\begin{thebibliography}{BFM{\etalchar{+}}24}

\bibitem[Ago04]{agol2004tameness}
Ian Agol.
\newblock Tameness of hyperbolic 3-manifolds.
\newblock {\em ArXiv:math/0405568}, 2004.

\bibitem[BFM{\etalchar{+}}24]{battista2023bounded}
Ludovico Battista, Stefano Francaviglia, Marco Moraschini, Filippo Sarti, and
  Alessio Savini.
\newblock Bounded cohomology classes of exact forms.
\newblock {\em Proc. Amer. Math. Soc.}, 152(1):71--80, 2024.

\bibitem[BG88]{bg}
Jean Barge and \'{E}tienne Ghys.
\newblock Surfaces et cohomologie born\'{e}e.
\newblock {\em Invent. Math.}, 92(3):509--526, 1988.

\bibitem[BM19]{BMentropy}
Michael Brandenbursky and Micha\l~\space Marcinkowski.
\newblock Entropy and quasimorphisms.
\newblock {\em J. Mod. Dyn.}, 15:143--163, 2019.

\bibitem[BM22]{BM}
Michael Brandenbursky and Micha\l~\space Marcinkowski.
\newblock Bounded cohomology of transformation groups.
\newblock {\em Math. Ann.}, 382(3-4):1181--1197, 2022.

\bibitem[CG06]{calegari-gabai}
Danny Calegari and David Gabai.
\newblock Shrinkwrapping and the taming of hyperbolic 3-manifolds.
\newblock {\em J. Amer. Math. Soc.}, 19(2):385--446, 2006.

\bibitem[Che20]{Chen}
Lei Chen.
\newblock Vanishing of the {E}uler class in {P}ower subgroups of the punctured
  mapping class group.
\newblock {\em ArXiv:2002.06729}, 2020.

\bibitem[EK71]{edwards-kirby}
Robert~D. Edwards and Robion~C. Kirby.
\newblock Deformations of spaces of imbeddings.
\newblock {\em Ann. of Math. (2)}, 93:63--88, 1971.

\bibitem[Fri17]{frigerio}
Roberto Frigerio.
\newblock {\em Bounded cohomology of discrete groups}, volume 227 of {\em
  Mathematical Surveys and Monographs}.
\newblock American Mathematical Society, Providence, RI, 2017.

\bibitem[GG04]{MR2104597}
Jean-Marc Gambaudo and {\'E}tienne Ghys.
\newblock Commutators and diffeomorphisms of surfaces.
\newblock {\em Ergodic Theory Dynam. Systems}, 24(5):1591--1617, 2004.

\bibitem[Gro82]{gromov82}
Michael Gromov.
\newblock Volume and bounded cohomology.
\newblock {\em Inst. Hautes \'{E}tudes Sci. Publ. Math.}, (56):5--99 (1983),
  1982.

\bibitem[Hub12]{Huber}
Thomas Huber.
\newblock Rotation quasimorphisms for surfaces.
\newblock {\em Ph.D thesis, ETH}, 2012.

\bibitem[JR21]{JekelRolland}
Solomon Jekel and Rita~Jim\'enez Rolland.
\newblock On the {N}on-vanishing of the {P}owers of the {E}uler {C}lass for
  {M}apping {C}lass {G}roups.
\newblock {\em Arnold Math J.}, 7:159–168, 2021.

\bibitem[Kat92]{skatok}
Svetlana Katok.
\newblock {\em Fuchsian groups}.
\newblock Chicago Lectures in Mathematics. University of Chicago Press,
  Chicago, IL, 1992.

\bibitem[Kim20]{kimura}
Mitsuaki Kimura.
\newblock Gambaudo--{G}hys construction on bounded cohomology.
\newblock {\em ArXiv:2009.00124}, 2020.

\bibitem[Mit84]{mitsumatsu}
Yoshihiko Mitsumatsu.
\newblock Bounded cohomology and {$l^1$}-homology of surfaces.
\newblock {\em Topology}, 23(4):465--471, 1984.

\bibitem[Mos65]{Moser}
J\"{u}rgen Moser.
\newblock On the volume elements on a manifold.
\newblock {\em Trans. Amer. Math. Soc.}, 120:286--294, 1965.

\bibitem[MT19]{mann.survey}
Kathryn Mann and Bena Tshishiku.
\newblock Realization problems for diffeomorphism groups.
\newblock In {\em Breadth in contemporary topology}, volume 102 of {\em Proc.
  Sympos. Pure Math.}, pages 131--156. Amer. Math. Soc., Providence, RI, 2019.

\bibitem[Nit]{nitsche}
Martin Nitsche.
\newblock Higher-degree bounded cohomology of transformation groups.
\newblock {\em ArXiv:2105.08698}.

\bibitem[Som97]{soma}
Teruhiko Soma.
\newblock Bounded cohomology and topologically tame {K}leinian groups.
\newblock {\em Duke Math. J.}, 88(2):357--370, 1997.

\bibitem[Thu22a]{thurston}
William~P. Thurston.
\newblock {\em The {G}eometry and {T}opology of {T}hree-{M}anifolds. {V}ol.
  {IV}}.
\newblock American Mathematical Society, Providence, RI, [2022] \copyright
  2022.
\newblock Edited and with a preface by Steven P. Kerckhoff and a chapter by J.
  W. Milnor.

\bibitem[Thu22b]{thurstonII}
William~P. Thurston.
\newblock Hyperbolic structures on 3-manifolds, {II}: surface groups and
  3-manifolds which fiber over the circle.
\newblock pages 79--110. Amer. Math. Soc., Providence, RI, [2022] \copyright
  2022.
\newblock August 1986 preprint, January 1998 eprint.

\end{thebibliography}

\end{document}